\documentclass[10pt]{amsart}
\usepackage{graphicx}
\usepackage{amscd}
\usepackage{amsmath}
\usepackage{amsthm}
\usepackage{amsfonts}
\usepackage{amssymb}
\usepackage{mathrsfs}
\usepackage{bm}
\usepackage{enumerate}
\usepackage{amsrefs}
\usepackage{xcolor}
\usepackage[colorlinks, citecolor=blue, linkcolor=red, pdfstartview=FitB]{hyperref}
\usepackage[latin1]{inputenc}

\usepackage{tikz}
\usepackage{tikz-cd}
\usepackage{caption}
\usetikzlibrary{decorations.markings}
\usepackage{lipsum}
\usepackage{mathrsfs}

\usepackage{marginnote}	
\usepackage{soul} 			

\newcommand{\bC}{{\mathbb{C}}}
\newcommand{\bD}{{\mathbb{D}}}

\newcommand{\bM}{{\mathbb{M}}}

\newcommand{\bT}{{\mathbb{T}}}

\renewcommand{\O}{{\mathcal{O}}}

  \newcommand{\R}{{\mathcal{R}}}

\newcommand{\fK}{{\mathfrak{K}}}

\newcommand{\rC}{\mathrm{C}}

\newcommand{\eps}{\varepsilon}
\renewcommand{\phi}{\varphi}

\newcommand{\upchi}{{\raise.35ex\hbox{$\chi$}}}

\newcommand{\ol}{\overline}

\newcommand{\mycomment}[1]{}

\newcommand{\qand}{\quad\text{and}\quad}

\newcommand{\ran}{\operatorname{ran}}

\newcommand{\supp}{\operatorname{supp}}

\newcommand{\Ann}{\operatorname{Ann}}

\newtheorem{lemma}{Lemma}[section]
\newtheorem{theorem}[lemma]{Theorem}
\newtheorem*{theorem*}{Theorem}
\newtheorem{proposition}[lemma]{Proposition}
\newtheorem{corollary}[lemma]{Corollary}
\newtheorem{theoremx}{Theorem}

\theoremstyle{definition}

\newtheorem{example}{Example}

\date{\today}

\author{Rapha\"el Clou\^atre}
\address{Department of Mathematics, University of Manitoba, Winnipeg, Manitoba, Canada R3T 2N2}
\email{raphael.clouatre@umanitoba.ca\vspace{-2ex}}

\author{Poornendu Kumar}
\email{Poornendukumar@gmail.com \vspace{-2ex}}

\thanks{R.C. was partially supported by an NSERC Discovery Grant. P.K. was partially supported by a PIMS postdoctoral fellowship.}



\title[Agler--McCarthy spectral varieties in the bidisc]{Annihilating ideals and Agler--McCarthy spectral varieties in the bidisc}
\begin{document}
\begin{abstract}
The closed unit bidisc $\ol{\bD^2}$ is known to be a spectral set for any pair $(T_1,T_2)$ of commuting contractions. When each $T_i$ is pure and has finite defect, the pair  admits a much smaller spectral set: the closure of a distinguished variety $V\subset \bD^2$. We exhibit conditions on $(T_1,T_2)$ that guarantee that no proper closed subset of $\ol{V}$ can be a spectral set. In addition, we examine the relationship between $V$ and  the annihilating ideal $\Ann(T_1,T_2)$ in $H^\infty(\bD^2)$. While $V$ is typically strictly larger than the zero set of $\Ann(T_1,T_2)$, we isolate a natural constrained  co-extension $(S_1,S_2)$ of $(T_1,T_2)$ whose Taylor spectrum is contained in $V$ and is closely linked to the so-called support of $\Ann(T_1,T_2)$.  We also find simple conditions on $\Ann(T_1)$ detecting whether $\Ann(T_1,T_2)$ is the ideal of functions vanishing on the joint point spectrum of $(S_1^*,S_2^*)$. 
\end{abstract}
\maketitle

\section{Introduction}

Let $T_1,T_2$ be commuting bounded linear operators on some Hilbert space $H$. We  denote the  Taylor spectrum  of the pair  by $\sigma(T_1,T_2)$. 
Let $X\subset \bC^2$ be a compact subset containing $\sigma(T_1,T_2)$, and let $\rC(X)$ denote the $\rC^*$-algebra of continuous complex-valued functions on $X$.  Let $\R(X)\subset \rC(X)$ denote the subalgebra of restrictions to $X$ of rational functions with poles outside of $X$. In this way, $\R(X)$ is an operator algebra.  We also let $\O(X)$ denote the space of functions that are holomorphic on a neighbourhood of $X$. The Taylor functional calculus \cite[Theorem 30.7]{muller2007} (see also the survey \cite{Cur}) then gives a well-defined homomorphism
\[
\O(X)\to B(H)
\]
\[
f\mapsto f(T_1,T_2), \quad f\in \O(X).
\]
We say that $X$ is a \emph{spectral set} for $(T_1,T_2)$ if this homomorphism is contractive on $\R(X)$. When this homomorphism is completely contractive on $\R(X)$ (i.e. contractive on every matrix level $\bM_n(\R(X))$) then we say that $X$ is a \emph{complete spectral set} for $(T_1,T_2)$.

Because they lie at the interface of function theory and operator theory, spectral sets have long been of interest; see  \cite{paulsen2002} and \cite[Chapter 37]{handbookLA2014} and the references therein. They still remain relevant to this day, being at the heart of the Crouzeix conjecture for instance \cite{crouzeix2004},\cite{crouzeix2007}.

It is easy to show that any spectral set contains another spectral set that is minimal with respect to inclusion. The question of identifying such a minimal spectral set, however, is more delicate. In the case of a jointly subnormal pair, the Taylor spectrum is a spectral set \cite{putinar1984} contained in every other spectral set, so it is clearly minimal. Beyond the subnormal case, fairly little is known regarding minimal spectral sets. In the single-variable situation, what appears to be the most definitive result in this direction was obtained a long time ago by Williams \cite{williams1967} (the notion of a spectral set for a single operator is defined completely analogously to the notion defined above for pairs). On the other hand, the multivariate case does not seem to have garnered much attention. 

Suppose now that $T_1$ and $T_2$ are both contractions. Let $\bD\subset \bC$ denote the open unit disc and let $\bT\subset \bC$ denote the unit circle. It is well known that $\ol{\bD^2}$ is a spectral set for $(T_1,T_2)$, by virtue of the classical inequality of Ando: rational functions with poles off $\ol{\bD^2}$ are uniform limits of polynomials. Interestingly, in many cases a much smaller set than $\ol{\bD^2}$ suffices.

Let $E$ be a finite-dimensional Hilbert space and let $\Psi:\bD \to B(E)$ be a contractive holomorphic function. We put
\[
V_\Psi=\{(z,w)\in \bD^2:\det(\Psi(z)-wI)=0\}.
\]
We say that $V_\Psi$ is an \emph{Agler--McCarthy variety} for the pair $(T_1,T_2)$ if 
\begin{equation}\label{Eq:AM}
\|p(T_1,T_2)\|\leq \sup_{V_\Psi} |p| \quad \text{ for every polynomial } p.
\end{equation}
When both $T_1$ and $T_2$ are pure contractions with finite defects, there is a rational matrix-valued pure inner function $\Psi$ such that $V_\Psi$ is an Agler--McCarthy variety for $(T_1,T_2)$.  This was first proved for pairs of matrices in a seminal paper of Agler--McCarthy \cite[Theorem 3.1]{AM2005}, and later extended to the infinite-dimensional setting  in \cite[Theorem 4.3]{DS2017}. Some of these details appear at the beginning of Section \ref{S:specsynth}. Importantly, it follows from  \cite[Theorem 1.12]{AM2005} (see also \cite{BKS2022, knese2010}) that $\Psi$ is a rational matrix-valued pure inner function precisely when $V_\Psi$ is a so-called  \emph{distinguished} variety, in the sense that $V_\Psi\cap \bD^2$ is non-empty and $\ol{V_\Psi}$ exits $\bD^2$ through its distinguished boundary $\bT^2$, i.e $\ol{V}\cap \partial \bD^2=\bT^2$.

The study of distinguished varieties goes back at least to Rudin's 1969 work \cite{RudinTAMS}. A landmark contribution is the work of Agler--McCarthy \cite{AM2005}. Since then, these varieties have garnered significant interest, due to their deep connections with several branches of mathematics. For example, distinguished varieties arise naturally as uniqueness sets for solvable Nevanlinna--Pick interpolation problems on $\mathbb{D}^2$, and they play an important role in strengthening Ando's inequality, as discussed above. In addition, building on Fedorov's work on unramified point-separating pairs of inner functions, Vegulla \cite[Theorem 3.4.4]{VegullaThesis} proved that every planar domain bounded by piecewise analytic curves is conformally equivalent to a distinguished variety, thereby providing a unified framework for understanding the function theory of multi-connected domains. More recently, \cite{GWZ} established a connection between distinguished varieties and the theory of essentially normal quotient modules, particularly in relation to the Arveson--Douglas conjecture. Furthermore, these varieties also appear in generalizations of Wermer's maximality theorem on the bidisc \cite[Theorem 1.2]{SWInventiones}.

Although it is not immediately clear from \eqref{Eq:AM}, we shall see below in Corollary \ref{C:AMspectral} that if $V_\Psi$ is an Agler--McCarthy distinguished variety for a pair $(T_1,T_2)$, then $\ol{V_\Psi}$  is in fact a spectral set for $(T_1,T_2)$.  In particular, this means that $\ol{\bD^2}$ is \emph{not} a minimal spectral set in this case.  The first goal of the paper is to find conditions on $(T_1,T_2)$ under which Agler--McCarthy are  minimal spectral sets. We undertake this task in Section \ref{S:minspec}.  As a first step, we take  another look at the univariate setting. Given a nice domain $G\subset \bC$ such that $\ol{G}$ is a spectral set for a single operator, we extend results from  \cite{williams1967} and find a sufficient condition on the functional calculus for minimality of $\ol{G}$ as a spectral set, provided that the spectrum is contained in $G$ and is not too pathological (see Theorem \ref{T:Williams}). In Theorem \ref{T:ratShilov}, we show how the spectral containment assumption can be recast in $\rC^*$-algebraic terms, using ideas from noncommutative Choquet theory. As an application, we show in Corollary \ref{C:Cuntz} that there is a contraction generating the $\rC^*$-algebra of compact operators for which $\ol{\bD}$ is a minimal spectral set. 

Turning back to the bivariate case, we obtain our first main result (Theorem \ref{T:min2var}), which we summarize below.  
Given a complex-valued function $f$ defined on a set $X$, throughout the paper we denote
\[
\|f\|_X=\sup_{x\in X}|f(x)|.
\]

\begin{theoremx}\label{T:A}
Let $T_1,T_2$ be commuting Hilbert space  contractions. Let $V\subset \bD^2$ be an Agler--McCarthy distinguished variety for $(T_1,T_2)$ with the following properties:
\begin{enumerate}[{\rm (a)}]
\item there is an open neighbourhood $B$ of $\ol{\bD^2}$ and a holomorphic function $\Gamma:B\to \bC$ such that $\Gamma|_{V}$ is a biholomorphism $V\to \bD$,
\item the spectrum of $(T_1,T_2)$ is contained in $V$, and
\item there is a non-constant function $\phi\in \O(V)$ such that $1=\|\phi\|_{V}=\|\phi(T_1,T_2)\|$.
\end{enumerate}
Let $K\subset \ol{V}$ be a proper closed subset. Then, $K$ is not a spectral set for $(T_1,T_2)$.
\end{theoremx}

As far as we know, outside the subnormal setting this is the first such minimality result in the context of a 
 \emph{pair} of commuting operators. 

Our next objective is to identify an Agler--McCarthy  variety (or a portion thereof) in terms of data uniquely determined by $T$. For $1\leq d\leq 2$, we denote by $H^\infty(\bD^d)$ the Banach algebra of bounded holomorphic functions on $\bD^d$. Henceforth, $T_1$ and $T_2$ will be assumed to be pure and to have finite defects.
In particular, both $T_1$ and the pair $(T_1,T_2)$ admit a weak-$*$ continuous $H^\infty$-functional calculus (see \eqref{Eq:functcalcT} below). We define
\[
\Ann(T_1)=\{f\in H^\infty(\bD):f(T)=0\} \qand
\]
\[
\Ann(T_1,T_2)=\{f\in H^\infty(\bD^2):f(T_1,T_2)=0\}.
\]
We let $Z(\Ann(T_1,T_2))\subset \bD^2$ be the zero set of $\Ann(T_1,T_2)$. Following \cite{AM2005} and \cite{DS2017}, in Section \ref{S:specsynth} we recall the construction of a rational matrix-valued pure inner function $\Psi$ such that $V_\Psi$ is an Agler--McCarthy distinguished variety for $(T_1,T_2)$.

%

We begin our investigation with the observation that
	\[
	Z(\Ann(T_1,T_2)) \subset V_\Psi
	\]
	(see \eqref{Eq:ZannV}). Easy examples reveal that the inclusion is strict in general. It follows from the construction of $\Psi$  that $V_\Psi$ is  the point spectrum of some isometric co-extension of the pair $(T_1,T_2)$. This suggests that there may be a more suitable  co-extension that more accurately reflects the algebraic properties of $(T_1,T_2)$. Indeed, we construct the \emph{constrained co-extension} $(S_1^\Psi, S_2^\Psi)$ of $(T_1,T_2)$, acting on the space $K_\Psi$ (see \eqref{Eq:S}). We then show that the set $\Omega_\Psi$, defined as the complex conjugate of the joint point spectrum of the pair $(S_1^{\Psi*}, S_2^{\Psi*})$, coincides with $Z(\Ann(T_1,T_2))$ (Theorem \ref{T:ZAnn}). We refine this result in Theorem \ref{T:supp}, which includes additional boundary information and describes the so-called support of $\Ann(T_1,T_2)$  \cite{CT2023} in terms of $V_\Psi$ and the full Taylor spectrum of $(S^\Psi_1,S^\Psi_2)$. We summarize these findings in the following.

\begin{theoremx}\label{T:C}
We have 
\[
\sigma(S^\Psi_1,S^\Psi_2)\subset \supp(\Ann(T_1,T_2))\subset \ol{V_\Psi}
\]
and
\[
\sigma(S^\Psi_1,S^\Psi_2)\cap \bD^2=\supp(\Ann(T_1,T_2))\cap \bD^2=Z(\Ann(T_1,T_2)).
\]
\end{theoremx}

This result extends to the bidisc some relationships between the support of the annihilating ideal and the spectrum that are known to hold in the disc and the ball \cite{bercovici1988},\cite{CT2023}.

The next natural question is to relate $\Ann(T_1,T_2)$ with the ideal $I(\Omega_\Psi)$ of functions in $H^\infty(\bD^2)$ vanishing on $\Omega_\Psi$. Our third main result elucidates this issue under a mild assumption on the pair $(T_1,T_2)$ (see Theorem \ref{T:specsynth}).

\begin{theoremx}\label{T:B}
 Assume that $\Ann(T_1)$ is non-zero and let $d\geq 0$ be the defect of $T_1$. Consider the following statements.
 \begin{enumerate}[{\rm (i)}]
\item If $\Theta\in H^\infty(\bD;B(\bC^d))$ is an inner function such that $K_\Psi=(\ran M_\Theta)^\perp$, then $\det\Theta$ is a Blaschke product with simple roots.
\item The joint eigenvectors of $S_1^{\Psi *}$ and $S_2^{\Psi *}$ span a dense subspace in $K_\Psi$.
\item $\Ann(T_1,T_2)=I(Z(\Ann(T_1,T_2)))=I(\Omega_\Psi)$.
\item $\Ann(T_1)$ is a radical ideal in $H^\infty(\bD)$.
\item $\Ann(T_1)$ is generated by a Blaschke product with simple roots.
\end{enumerate}
Then, we have
\[
{\rm (i)}\Rightarrow {\rm (ii)}\Rightarrow{\rm (iii)}\Rightarrow{\rm (iv)}\Leftrightarrow{\rm (v)}.
\]
\end{theoremx}
Notably, this result connects the spectral synthesis problem for a pair of commuting operators (see  \cite{BaranovBelov} and the references therein) to  the algebraic structure of the annihilators.

\subsection*{AI declaration} AI tools were used only for literature search, as well as to compute the polynomial defining the variety $W$ in Example \ref{E:sc}. All other mathematical ideas and proofs were obtained by the authors without AI assistance.

\section{Absence of proper spectral subsets}\label{S:minspec}

\subsection{Preliminaries on spectral sets}
Here and throughout, $B(H)$ denotes the $\rC^*$-algebra of bounded linear operators on a Hilbert space $H$.  Given finitely many commuting operators $T_1,\ldots,T_d \in B(H)$, we let $T=(T_1,\ldots,T_d)$. Let $\sigma(T)\subset \bC^d$ denote the Taylor spectrum of $T$.  A compact set $X\subset \bC^d$ containing $\sigma(T)$ is a \emph{spectral set} for $T$ if
\[
\|r(T)\|\leq \|r\|_X, \quad r\in \R(X).
\]

We start  by recording some  properties regarding the Riesz--Dunford holomorphic functional calculus and spectral sets for single operators. For this purpose, note that when $d=1$, $\sigma(T)$ is the usual spectrum of $T$ relative to $B(H)$. As is well known \cite[Theorem 4.28]{douglas1998}, for any unital $\rC^*$-subalgebra $A\subset B(H)$ we have that $\sigma(T)$ consists of those $\lambda\in \bC$ for which $T-\lambda 1$ is not invertible in $A$. 

\begin{lemma}\label{L:homom}
Let $T\in B(H)$ and let $A\subset B(H)$ be a unital $\rC^*$-subalgebra. Let $\pi:A\to B(K)$ be a unital $*$-homomorphism. Then, the following statements hold.
\begin{enumerate}[{\rm (i)}]
\item  Let $f\in\O(\sigma(T))$. Then, $\pi(f(T))=f(\pi(T))$.
\item A spectral set for $T$ is also a spectral set for $\pi(T)$.
\end{enumerate}
\end{lemma}
\begin{proof}
We make a preliminary observation. Because $\pi$ is unital, the spectrum of $\pi(T)$  relative to $\pi(A)$ is contained in the spectrum  of $T$ relative to $A$. Since $\pi(A)$ is a $\rC^*$-subalgebra of $B(K)$, by the observation preceding the lemma, this means that $\sigma(\pi(T))\subset \sigma(T)$.

(i) The equality is easily verified when $f\in \R(\sigma(T))$. The general case follows from Runge's approximation theorem along with the continuity property of the Riesz--Dunford functional calculus \cite[Theorem I.37]{muller2007}.

(ii) Let $X\subset \bC$ be a compact set that is spectral for $T$. Then, $X$ contains $\sigma(T)$, and thus $\sigma(\pi(T))$ as well. For $f\in \R(X)$, we find $f\in \O(\sigma(T))$, and we may apply (i) to see that
\[
\|f(\pi(T))\|=\|\pi(f(T))\|\leq \|f(T)\|\leq \|f\|_X.\qedhere
\]
\end{proof}

\begin{lemma}\label{L:specimprov}
Let $G\subset \bC$ be a bounded, connected, open set with boundary consisting of finitely many rectifiable simple closed curves. Let $T\in B(H)$ be an operator admitting $\ol{G}$ as a spectral set and such that $\sigma(T)\subset G$. Then, for every $f\in \O(G)$ we have $\|f(T)\|\leq \|f\|_G$.
\end{lemma}
\begin{proof}
Let $f\in \O(G)$. By \cite[Theorem 3]{AS1967} (see also \cite{fisher1969} for a more general version), there is a sequence $(r_n)$ of rational functions with poles off $\ol{G}$ converging pointwise to $f$ on $G$ and such that $\|r_n\|_G\leq \|f\|_G$. By Montel's theorem, upon passing to a subsequence if necessary we can also assume that $(r_n)$ converges uniformly to $f$ on a neighbourhood of $\sigma(T)$, so that $f(T)$ is the norm limit of $(r_n(T))$ by \cite[Theorem I.37]{muller2007}. Using that $\ol{G}$ is a spectral set for $T$, we obtain
\[
\|f(T)\|=\lim_{n\to\infty}\|r_n(T)\|\leq \limsup_{n\to\infty}\|r_n\|_G\leq \|f\|_G.\qedhere
\]
\end{proof}

Next, we aim to establish multivariate variants of Lemma \ref{L:specimprov}. For this purpose, we recall some terminology.  Let $d\geq 1$ be an integer and let $K\subset \bC^d$ be a compact set. We say that $K$ is \emph{polynomially convex} if, whenever $z\in \bC^d$ satisfies
\[
|p(z)|\leq \|p\|_K \quad \text{ for every polynomial } p,
\]
it follows that $z\in K$. 
We say that $K$ is \emph{rationally convex} if
whenever $z\in \bC^d$ satisfies
\[
p(z)\in p(K) \quad \text{ for every polynomial } p,
\]
it follows that $z\in K$. Alternatively, \cite[Statement (1.1)]{Stolzenberg1963} shows that rational convexity of $K$ is equivalent to the following: whenever $z\in \bC^d$ satisfies
\[
|r(z)|\leq \|r\|_K , \quad r\in \R(K)
\]
it follows that $z\in K$.  We now need a basic topological fact for which we lack an exact reference.

\begin{lemma}\label{L:ratconvnhbd}
Let $K\subset \bC^d$ be a compact set and let $U\subset \bC^d$ be an open set containing $K$. 
\begin{enumerate}[{\rm (i)}]
\item If $K$ is polynomially convex, then there is a compact polynomially convex set $K'\subset U$ containing $K$ in its interior.
\item If $K$ is rationally convex, then there is a compact rationally convex set $K''\subset U$ containing $K$ in its interior.
\end{enumerate}
\end{lemma}
\begin{proof}
We adapt the argument given in \cite[Lemma 2.7.4]{hormander1990}.
There is $M>0$ large enough so that $\|z_j\|_{K}<M$ for $1\leq j \leq d$. Then,  
\[
B=\{z\in \bC^d:|z_j|\leq M, 1\leq j\leq d\}
\]
is a compact set containing $K$. Consider the compact set $L=(\bC^d\setminus U)\cap B$. 

(i) For each $w\in L$, we have that $w\notin K$, so by polynomial convexity there is a polynomial $p_w$ and $t_w>0$ such that $|p_w(w)|>t_w>\|p_w\|_K$.  Hence
\[
L\subset \bigcup_{w\in L} \{z\in \bC^d: |p_w(z)|>t_w\}.
\]
Each set $\{z\in \bC^d: |p_w(z)|>t_w\}$ is open, so by compactness of $L$ there are polynomials $p_1,\ldots, p_m$ and $t_1,\ldots,t_m>0$ such that $\|p_j\|_K<t_j$ for each $1\leq j\leq m$ and
\[
L\subset \bigcup_{j=1}^m \{z\in \bC^d: |p_j(z)|>t_j\}.
\]
Define the set
\[
K'=\{z\in \bC^d: |p_j(z)|\leq t_j , 1\leq j\leq m\}\cap  B
\]
which is clearly compact, contains $K$ its interior, and is polynomially convex. 
If $z\in K'$, then $z\in B$ but $z\notin \bigcup_{j=1}^m \{z\in \bC^d: |p_j(z)|>t_j\}$ so $z\in B\setminus L\subset U$, whence $K'\subset U$ as desired.

(ii)  For each $w\in L$, we have that $w\notin K$, so by rational convexity there is a polynomial $p_w$ such that $p_w(w)\notin p_w(K)$. Hence
\[
L\subset \bigcup_{w\in L} \{z\in \bC^d: p_w(z)\notin  p_w(K)\}.
\]
Each set $\{z\in \bC^d: p_w(z)\notin  p_w(K)\}$ is open, so by compactness of $L$ there are polynomials $p_1,\ldots, p_m$ and $\eps>0$ such that
\[
L\subset \bigcup_{j=1}^m \{z\in \bC^d: d(p_j(z),p_j(K))>\eps\}.
\]
Define the set
\[
K''=\{z\in \bC^d: d(p_j(z),p_j(K))\leq \eps, 1\leq j\leq m\}\cap B
\]
which is clearly compact, contains $K$ its interior, and is rationally convex. If $z\in K''$, then $z\in B$ but $z\notin \bigcup_{j=1}^m \{z\in \bC^d:d(p_j(z),p_j(K))>\eps\}$ so $z\in B\setminus L\subset U$, whence $K''\subset U$ as desired.
\end{proof}

We now arrive at a key technical result.

\begin{theorem}\label{T:specimprov}
Let  $T_1,\ldots,T_d\in B(H)$ be commuting operators and let $K\subset \bC^d$ be a compact set containing $\sigma(T)$. Assume that either of the following holds.
\begin{enumerate}[{\rm (i)}]
\item $K$ is polynomially convex and $\|p(T)\|\leq \|p\|_K$ for every polynomial $p$.
\item $K$ is a rationally convex spectral set for $T$.
\end{enumerate}
Then, for every $f\in \O(K)$ we have $\|f(T)\|\leq \|f\|_K$.
\end{theorem}
\begin{proof}
Let $f\in \O(K)$. Then, there is an open set $U\subset \bC^d$ on which $f$ is holomorphic. 

(i) Applying Lemma \ref{L:ratconvnhbd}, we can find a polynomially convex compact set $K'$ contained in $U$ and containing $K$ in its interior. Apply the Oka--Weil theorem \cite[Theorem 1.5.1]{stout2007} to find a sequence of polynomials $(p_n)$ converging uniformly to $f$ uniformly on $K'$.  The  continuity of the Taylor functional calculus \cite[Theorem 30.7]{muller2007} implies  that the sequence $(p_n(T))$ converges in norm to $f(T)$. Therefore,
\begin{align*}
\|f(T)\|&=\lim_{n\to\infty}\|p_n(T)\| \leq\lim_{n\to\infty}\|p_n\|_K= \|f\|_K.
\end{align*}

(ii) Applying Lemma \ref{L:ratconvnhbd}, we can find a rationally convex compact set $K''$ contained in $U$ and containing $K$ in its interior. By the rational version of the Oka--Weil theorem \cite[page 44]{stout2007},  there is a sequence of rational functions $(r_n)$ with poles off $K''$ converging uniformly to $f$  on $K''$. 
The  continuity of the Taylor functional calculus \cite[Theorem 30.7]{muller2007}  implies that the sequence $(r_n(T))$ converges in norm to $f(T)$. Therefore,
\begin{align*}
\|f(T)\|&=\lim_{n\to\infty}\|r_n(T)\| \leq\lim_{n\to\infty}\|r_n\|_K= \|f\|_K.
\end{align*}
\end{proof}
\subsection{Minimal spectral sets in one variable}
Next, we give a criterion for minimality of spectral sets in one variable. This extends a result of Williams \cite[Theorem 9]{williams1967} past the setting where $H$ is finite-dimensional, or where $T$ is compact. Our argument largely follows the scheme of  proof thereof. 


\begin{theorem}\label{T:Williams} 
Let $G\subset \bC$ be a bounded, connected, open set with boundary consisting of finitely many rectifiable simple closed curves. Let $T\in B(H)$ be an operator admitting $\ol{G}$ as spectral set.
Assume that
\begin{enumerate}[{\rm (i)}]
\item $\sigma(T)$ is contained in $G$,
\item given an open ball $B\subset G$ with closure disjoint from $\sigma(T)$, there is an open set $U\subset G\setminus \ol{B}$ containing $\sigma(T)$ with boundary consisting of finitely many rectifiable simple closed curves,
\item there is a non-constant function $\phi\in \O(G)$ such that $\|\phi(T)\|=\|\phi\|_G=1$.
\end{enumerate}
Then, $\ol{G}$ is a minimal spectral set for $T$.
\end{theorem}
\begin{proof}
Let $X_0 \subset \ol{G}$ be a proper closed subset which is a spectral set for $T$.  Note that $X_0$ does not contain $G$, so that  $X_0 \subset \ol{G}\setminus B$, where $B\subset G$ is an open ball with closure disjoint from $\sigma(T)$. 
If we set $G_0=G\setminus \ol{B}$, then $G_0$ is a bounded, connected, open set with boundary consisting of finitely many rectifiable simple closed curves. Furthermore, $\sigma(T)$ is contained in $G_0$, and the set $X=\ol{G_0}$ is spectral for $T$.
By assumption, there is an open set $U\subset G_0$ containing $\sigma(T)$ with boundary consisting of finitely many rectifiable simple closed curves.

Next, by basic $\rC^*$-algebra theory, there is an irreducible $*$-representation $\pi:B(H)\to B(K)$ along with unit vectors $\xi,\eta\in K$ such that 
\begin{equation}\label{Eq:extphi}
1=\|\phi\|_G=\|\phi(T)\|=\langle \pi(\phi(T))\xi,\eta\rangle=\langle \phi(\pi(T))\xi,\eta\rangle
\end{equation}
where the last equality follows from Lemma \ref{L:homom}. 
By another application of Lemma \ref{L:homom}, we see that $X_0$ is also a  spectral set for $\pi(T)$. Moreover, $\sigma(\pi(T))\subset U$.

Define the function $\omega:\bC\setminus \sigma(\pi(T))\to \bC$ as
\[
\omega(z)=\frac{1}{2\pi i}\langle \pi(zI-T)^{-1}\xi,\eta\rangle, \quad z\in \bC\setminus \sigma(\pi(T)).
\]
Let $X^\circ$ denote the interior of $X$, so that $X^\circ=G_0$. Consider the Hardy space $E^1(X^\circ)$ (see \cite[Chapter 10]{duren1970}). Denote by $H^\infty(X^\circ)$ the Banach algebra of bounded holomorphic functions on $X^\circ$.
For $g\in H^\infty(X^\circ)$, note that
\[
 \int_{\partial X} g(z) \omega(z) dz=\langle g(\pi(T))\xi,\eta\rangle
\]
by properties of the Riesz--Dunford functional calculus \cite[page 99]{williams1967}. Lemma \ref{L:specimprov} implies that 
\[
\left|  \int_{\partial X} g(z) \omega(z) dz\right|\leq \|g(\pi(T))\| \leq\|g\|_{X} 
\]
for every $g\in H^\infty(X^\circ)$. Thus,  \eqref{Eq:extphi}  implies that the function $\phi$ is a solution to the extremal problem
\[
\sup\left\{ \left| \int_{\partial X} g(z) \omega(z) dz\right|: g\in H^\infty(X^\circ), \|g\|=1\right\}.
\]
By a duality result of Havinson \cite{havinson1957} analogous to \cite[Section 8.1]{duren1970}, there is $h\in E^1(X^\circ)$ such that
 \[
 \int_{\partial X}\phi(z)(\omega(z)-h(z))dz=1=\int_{\partial X}|\omega(z)-h(z)|d|z|.
 \]
 Therefore $\int_{\partial X}|\phi(z)(\omega(z)-h(z))|d|z|=1$ and
 \begin{equation}\label{Eq:phi}
  (1-|\phi|)|\omega-h|=0 \quad \text{ almost everywhere on } \partial X
 \end{equation}
 with respect to arc length measure (see  \cite[Theorem 5]{williams1967}). Observe that $|\phi|<1$ on $\partial B$ by the maximum modulus principle, since $\phi$ is non-constant on $G$. Thus, $\omega|_{\partial B}=h|_{\partial B}$ almost everywhere by \eqref{Eq:phi}. 
 
Consider $Y=X\setminus U$, which satisfies $\partial Y=\partial U\cup \partial B\cup \partial G$, so the boundary of $Y$ consists of finitely many rectifiable simple closed curves.  Then, both $\omega|_{Y^\circ}$ and $h|_{Y^\circ}$ belong to $E^1(Y^\circ)$. Since these functions agree almost everywhere on $\partial B$, the Riesz theorem (see \cite[Theorem 10.3]{duren1970} along with the discussion at the bottom of page 182) implies that they must agree almost everywhere on $\partial U$ as well. But this is impossible: on one hand we have
$
\int_{\partial U}\phi(z)h(z)dz=0
$ since $h\phi$ is holomorphic in  $U$, while on the other hand 
\[
\int_{\partial U}\phi(z)\omega(z)dz=\langle \phi(\pi(T))\xi,\eta\rangle=1
\]
because $\sigma(\pi(T))\subset U$. 
\end{proof}

We remark that the regularity property (ii) imposed on $\sigma(T)$ in Theorem \ref{T:Williams} is satisfied, for instance, when $\sigma(T)$ has only finitely many accumulation points.

Next, we will show how condition (i) above can be encoded as a $\rC^*$-algebraic property of $T$. For this purpose, we require some machinery from Arveson's theory of the noncommutative Choquet boundary \cite{arveson1969}.  

Let $A\subset B(H)$ be a unital subalgebra. A unital $*$-representation $\chi:\rC^*(A)\to \bC$ is said to be a \emph{boundary character} for $A$ if, given a unital contractive linear map $\psi:\rC^*(A)\to \bC$ satisfying $\chi|_A=\psi|_A$, we must have $\chi=\psi$ on $\rC^*(A)$. Recall also that the \emph{numerical range} of an operator $T\in B(H)$ is the set 
\[
W(T)=\{(Th,h):h\in H, \|h\|=1\}.
\]

\begin{theorem}\label{T:ratShilov}
Let $T\in B(H)$. Let $X\subset \bC$ be a spectral set for $T$ with the property that $\ol{\R(X)}$ has closed Choquet boundary in $\rC(X)$. Let $A\subset B(H)$ denote the unital subalgebra generated by $T$, and let $A_X=\{r(T):r\in \R(X)\}\subset B(H)$.
Then, the spectrum of $T$ is contained in the interior of $X$ provided that either of the following sets of conditions holds.
\begin{enumerate}[{\rm (a)}]
\item $X$ contains the numerical range of $T$ and $A$ admits no boundary character $\chi$ such that $\chi(T)\in \partial X$.
\item  $X$ is a complete spectral set for $T$ and $A_X$ admits no boundary character $\chi$ such that $\chi(T)\in \partial X$.
\end{enumerate}
\end{theorem}
\begin{proof}
First note that since $X$ is a spectral set for $T$, so it must contain $\sigma(T)$. Thus, it suffices to show that $\sigma(T)$ is disjoint from the boundary $\partial X$.  Let $\zeta\in \partial X\cap \sigma(T)$.

(a)  If $X$ contains the numerical range $W(T)$, then we must have $\zeta\in \sigma(T)\cap \partial W(T)$, since $\ol{W(T)}$ always contains $\sigma(T)$. Then, there is a boundary character $\chi$ for $A$ such that $\chi(T)=\zeta$ by \cite[Theorem 3.1.2]{arveson1969}.

(b) If $X$ is a complete spectral set for $T$, then the unital homomorphism $\rho_0:\R(X)\to B(H)$ defined as
\[
\rho_0(r)= r(T), \quad r\in \R(X)
\]
is completely contractive. By Arveson's extension theorem, there is a unital completely positive map $\rho:\rC(X)\to B(H)$ extending $\rho_0$.

Now, $T-\zeta I$ is not invertible in the commutative Banach algebra $\ol{A_X}$, so there is a character $\chi_0:A_X\to\bC$ satisfying $\chi_0(T)=\zeta$.  Let $\chi:B(H)\to\bC$ be a unital contractive linear map extending $\chi_0$. We will show that $\chi$ is necessarily a $*$-homomorphism on $\rC^*(A_X)$, which immediately implies that $\chi|_{\rC^*(A_X)}$ is a boundary character for $A_X$.

Note that 
\begin{equation}\label{Eq:Shilov}
(\chi\circ \rho)(r)=\chi_0(r(T))=r(\zeta), \quad r\in \R(X).
\end{equation}
Since $\zeta\in \partial X$, we see from \cite[Theorem II.1.3]{gamelin1969}  that $\zeta$ lies in the Shilov boundary of $\ol{\R(X)}$. In turn, because the Choquet boundary is assumed to be closed, it must coincide with the Shilov boundary \cite[Section II.11]{gamelin1969}, and hence $\zeta$ lies in the Choquet boundary. In particular, \eqref{Eq:Shilov} implies that $\chi\circ \rho$ is the character on $\rC(X)$ of evaluation at $\zeta$, whence $\chi\circ \rho$ is a $*$-homomorphism. For $r\in \R(X)$,  we find by the Schwarz inequality  \cite[Proposition 3.3]{paulsen2002} that 
\begin{align*}
\chi(r(T)^*r(T))&=\chi(\rho(r)^*\rho(r))\leq(\chi\circ \rho)(r^*r)\\
&=|(\chi\circ\rho)(r)|^2=|\chi(r(T))|^2\\
&\leq \chi(r(T)^*r(T))
\end{align*}
so $\chi(r(T)^*r(T))=|\chi(r(T))|^2$. Similarly $\chi(r(T)r(T)^*)=|\chi(r(T))|^2$, so that $\chi$ is a $*$-homomorphism by \cite[Theorem 3.18]{paulsen2002}.
\end{proof}

We can now give an application of Theorem \ref{T:Williams}. 

\begin{corollary}\label{C:Cuntz}
Let $\fK$ denote the $\rC^*$-algebra of compact operators on a separable Hilbert space. Then, $\fK$ is generated by a single contraction for which the closed unit disc is a minimal spectral set. 
\end{corollary}
\begin{proof}
It is well known that there is $T\in \fK$  with $\|T\|=1$ such that $\rC^*(T)=\fK$ \cite{nagisa2004}. It follows from von Neumann's inequality and Theorem \ref{T:specimprov} that $\ol{\bD}$ is a spectral set for $T$. Moreover, $\sigma(T)$ contains at most one accumulation point because $T$ is compact, so condition (ii) of Theorem \ref{T:Williams} is satisfied by the remark following that theorem. Since $\|T\|=\|z\|=1$, condition (iii) also holds. To see that $\ol{\bD}$ is a minimal spectral set for $T$, it thus suffices to check that the spectrum of $T$ is contained in $\bD$. By virtue of Theorem \ref{T:ratShilov}, it suffices in turn to ensure that the unital algebra generated by $T$ has no boundary character $\chi$ such that $|\chi(T)|=1$. But this is an immediate consequence of the fact that the only one-dimensional $*$-representation of $\fK+\bC 1$ annihilates all compact operators.
\end{proof}



\subsection{Minimality of Agler--McCarthy varieties}
We now apply our findings on minimality of spectral sets to Agler--McCarthy varieties. 
We need a preliminary statement that collects some known facts (see for instance \cite[Lemma 7.4]{BKS2022}).

\begin{lemma}\label{L:polyconv}
Let $\Psi$ be a rational matrix-valued  function on $\bD$ for which
\[
V=\{(z,w)\in \bD^2:\det(\Psi(z)-wI)=0\}
\]
is a distinguished variety.
Then, there is a polynomial $\psi$ in two variables such that
\[
\{(z,w)\in \ol{\bD}^2:\det(\Psi(z)-wI)=0\}=\{(z,w)\in \ol{\bD}^2:\psi(z,w)=0\}
\]
and moreover this set coincides with $\ol{V}$. In particular, $\ol{V}$ is polynomially convex and every closed subset $K\subset \ol{V}$ is rationally convex.
\end{lemma}
\begin{proof}
Rational inner functions such as $\Psi$ are continuous on $\ol{\bD}$, so the  set
\[
V'=\{(z,w)\in \ol{\bD}^2:\det(\Psi(z)-wI)=0\}
\]
is  closed. Then,  $\ol{V}\subset V'$. Conversely, let $(\lambda,\mu)\in  V'$,  so that $\det(\Psi(\lambda)-\mu I)=0$. Choose a sequence $(\lambda_n)$ in $\bD$ converging to $\lambda$. It follows that the sequence $(\det(\Psi(\lambda_n)-\mu I))_n$ converges to $0$. Hence, for each $n$ there is   $\mu_n\in \sigma(\Psi(\lambda_{n}))$ such that $(\mu_n)$ converges to $\mu$. Since $V$ is a distinguished variety, each $\mu_n$ lies in the open disc $\bD$. Hence, $((\lambda_n,\mu_n))_n$ is a sequence in $V$ converging to $(\lambda,\mu)$. We conclude that $V'=\ol{V}$.

Next, since $\Psi$ is rational, so is the function
\[
(z,w)\mapsto \det(\Psi(z)-wI)
\]
so there is a polynomial $\psi$ in two variables such that
\[
\ol{V}=\{(z,w)\in \ol{\bD}^2:\psi(z,w)=0\}.
\] 
The polynomial convexity statement now follows from a standard argument. 
Given $(\lambda,\mu)\in \bC^2$ such that 
\[
|p(\lambda,\mu)|\leq \max_{\ol{V}} |p|
\]
for every polynomial $p$, it readily follows that that $\psi(\lambda,\mu)=0$. In addition,
\[
|\lambda|\leq \|z_1\|_{\ol{V}}\leq 1  \qand |\mu|\leq \|z_2\|_{\ol{V}}\leq 1
\]
so that $(\lambda,\mu)\in \ol{V} $. We conclude that $\ol{V}$ is polynomially convex. 

Finally, let $K\subset \ol{V}$ be a closed subset. Then, $K\subset \{(z,w)\in \bC^2:\psi(z,w)=0\}$. If $\widehat K\subset \bC^2$ denotes the rational convex hull, then by \cite[(1.1)]{Stolzenberg1963} we see that $\widehat K\subset  \{(z,w)\in \bC^2:\psi(z,w)=0\}$. In turn, it is well known that $V$ is purely one-dimensional \cite[Proposition 4.2]{FG2002}, so
applying \cite[(1.12)]{Stolzenberg1963} we see that $K$ is rationally convex.
\end{proof}

We can now confirm that Agler--McCarthy varieties are indeed spectral sets.

\begin{corollary}\label{C:AMspectral}
Let $T_1,T_2\in B(H)$ be commuting contractions. Assume that there is a distinguished Agler--McCarthy  variety $V\subset \bD^2$ for $T=(T_1,T_2)$. Then, $\ol{V}$ is a spectral set for $T$ and for every $f\in \O(\ol{V})$ we have $\|f(T)\|\leq \|f\|_{\ol{V}}$. 
\end{corollary}
\begin{proof}
By Lemma \ref{L:polyconv}, there is a polynomial $\psi$ such that
\[
\ol{V}=\{(z,w)\in \ol{\bD}^2:\psi(z,w)=0\}.
\]
Note that $\psi(T_1,T_2)=0$ by \eqref{Eq:AM}. It then follows from the spectral mapping theorem \cite[Theorem 7.7]{muller2007} that $\sigma(T_1,T_2)$ is contained in the zero set of $\psi$ inside of $\ol{\bD}^2$, which is simply $\ol{V}$.  Since $\ol{V}$ is polynomially convex by Lemma \ref{L:polyconv}, we may apply Theorem \ref{T:specimprov} to find the remaining conclusions.
\end{proof}
Next, we need a consequence of the uniformization theorem in the context of varieties. It is likely known to experts, but we lack a proper reference.

\begin{lemma}\label{L:sc}
Let $V\subset \bD^2$ be a subset such that $\ol{V}\cap \partial\bD^2\subset \bT^2$. Assume that  there is a biholomorphism $\Phi:\bD\to V$. Then, there are an open neighbourhood $E$ of $\ol{\bD}$ and a holomorphic map $\Theta:E\to \bC^2$ such that $\Theta|_\bD=\Phi$ and $\Theta(\bT)\subset \bT^2$.
\end{lemma}
\begin{proof}
Write $\Phi=(\Phi_1,\Phi_2)$.  Assume that there is $\eps>0$ and a sequence $(z_n)\in \bD$ such that $|z_n|\to 1$ and $|\Phi_1(z_n)|\leq 1-\eps$. Then, the cluster points of the sequence $(\Phi_1(z_n),\Phi_2(z_n))$ must all lie in $\bD^2$, since $\ol{V}\cap \partial  \bD^2\subset \bT^2$. Upon applying $\Phi^{-1}$, this means that the sequence $(z_n)$ clusters at a point in $\bD$, which is absurd. We conclude that $|\Phi_1(z)|\to 1$ uniformly as $|z|\to 1$. 
A classical theorem of Fatou \cite[Theorem 4.1]{GMR2017} then implies that $\Phi_1$ is a finite Blaschke product. The same is true  of $\Phi_2$. Hence, $\Phi_1$ and $\Phi_2$ extend holomorphically to an open neighbourhood of $\ol{\bD}$. The existence of the map $\Theta$ follows immediately from this.
\end{proof}
We can now examine the minimality of Agler--McCarthy varieties as spectral sets.

\begin{theorem}\label{T:min2var}
Let $T_1,T_2\in B(H)$ be commuting contractions. Let $V\subset \bD^2$ be an Agler--McCarthy distinguished variety for $(T_1,T_2)$ with the following properties:
\begin{enumerate}[{\rm (a)}]
\item there is an open neighbourhood $B$ of $\ol{\bD^2}$ and a holomorphic function $\Gamma:B\to \bC$ such that $\Gamma|_{V}$ is a biholomorphism $V\to \bD$,
\item the spectrum of $(T_1,T_2)$ is contained in $V$, and
\item there is a non-constant function $\phi\in \O(V)$ such that $1=\|\phi\|_{V}=\|\phi(T_1,T_2)\|$.
\end{enumerate}
Let $K\subset \ol{V}$ be a proper closed subset. Then, $K$ is not a spectral set for $(T_1,T_2)$.
\end{theorem}
\begin{proof}
Throughout the proof, we write $T=(T_1,T_2)$.  Assume towards a contradiction that
$K$ is a spectral set for $T$. In particular,  the spectrum of $T$ is contained in $K$. By Lemma \ref{L:polyconv}, we see that $K$ is rationally convex, so we may then apply Theorem \ref{T:specimprov} to see
\begin{equation}\label{Eq:specO}
\|f(T)\|\leq \|f\|_K,\quad f\in \O(K).
\end{equation}
 Consider $R=\Gamma(T)$. Since the spectrum of $T$ is contained in $V$, the spectral mapping property \cite[Theorem 7.7]{muller2007} implies that the spectrum of $R$ is contained in $\Gamma(K\cap V)\subset \Gamma(K)\cap \bD$. 
We claim that $\Gamma(K)\subset \ol{\bD}$ is a spectral set for $R$. To see this, let $f\in \O(\Gamma(K))$. There is an open set $U\subset \bC$ containing $\Gamma(K)$ on which $f$ is holomorphic.  Then, the open set $\Gamma^{-1}(U)$ contains $K$ so $f\circ \Gamma\in \O(K)$ and $f(R)=(f\circ \Gamma|_{\Gamma^{-1}(U)})(T)$  \cite[Theorem 30.12]{muller2007}. In turn, \eqref{Eq:specO} implies
\[
\|f(R)\|=\|(f\circ \Gamma|_{\Gamma^{-1}(U)})(T)\|\leq \|f\circ \Gamma\|_K=\|f\|_{\Gamma(K)}.
\]
and the claim is established.

Next, by Lemma \ref{L:sc}, there are an open neighbourhood $E$ of $\ol{\bD}$ and a holomorphic map $\Theta:E\to \bC^2$ such that $\Theta|_\bD=(\Gamma|_V)^{-1}:\bD\to V$. Note that $\Gamma^{-1}(E)\subset B$ is an open set containing $V\cup K$.
It follows from \cite[Theorem 30.12]{muller2007} again that $\Theta(R)=(\Theta\circ \Gamma|_{\Gamma^{-1}(E)})(T)$. Write $\Theta=(\Theta_1,\Theta_2)$ and note that $\Theta_j\circ \Gamma|_{V}=z_j|_{V}$ for $1\leq j\leq 2$, by construction. By continuity it follows that $\Theta_j\circ \Gamma|_{{\ol{V}}}=z_j|_{\ol{V}}$ for $1\leq j\leq 2$.
Since $\Theta_j\circ \Gamma|_{\Gamma^{-1}(E)}-z_j\in \O(K)$, we find by \eqref{Eq:specO} that
\begin{align*}
\|T_j-\Theta_j(R)\|&=\|(z_j-\Theta_j\circ \Gamma|_{\Gamma^{-1}(E)})(T)\|\\
&\leq \|z_j-\Theta_j\circ \Gamma|_{\Gamma^{-1}(E)}\|_K=0
\end{align*}
whence $T=\Theta(R)$.   It follows that $\|(\phi\circ \Theta)(R)\|=\|\phi(T)\|=1$ while $\|\phi\circ \Theta\|_{\bD}=\|\phi\|_V=1$. By assumption on $\phi$, the function $\phi\circ \Theta$ is non-constant on $\bD$, so $\ol{\bD}$ is a minimal  spectral set for $R$ by Theorem \ref{T:Williams}. But this is absurd, since $\Gamma(K)$ is a proper subset of $\ol{\bD}$ that is a spectral set for $R$.
\end{proof}

Let $S\in B(H)$ be an operator with $\|S\|=1$ and spectrum contained in $\bD$. Consider the pair $T=(S,S)$ and the distinguished varieties 
\[
V_1=\{(z,w)\in \bD^2: z=w\} \qand V_2=\{(z,w)\in \bD^2: z^2=w^2\}.
\]
It readily follows from von Neumann's inequality that $V_1$ and $V_2$ are Agler--McCarthy varieties for $T$. Moreover, $V_2$ satisfies condition (c) from Theorem \ref{T:min2var}, yet $\ol{V_1}$ is a proper,  closed subset of $\ol{V_2}$ that is a spectral set. This shows that condition (a) above cannot simply be removed.

Next, we exhibit some instances of distinguished varieties where condition (a) can be verified. For instance, 
let $n\geq 1$ be a positive integer and consider the polynomial $p(z,w)=w^n-z$. Then,
\[
V=\{(z,w)\in \bD^2:p(z,w)=0\}=\{(w^n,w):w\in \bD\}
\]
is a distinguished variety, and the  holomorphic map $\Gamma:\bC^2\to \bC$ defined as $\Gamma(z,w)= w$ restricts to a biholomorphism $V\to\bD$.   We construct a more intricate example below.
\begin{example}\label{E:sc}
Let $U=\{z\in \bC:|z|<2\}$ and $b:U\to\bC$ defined as
\[
b(z)=\frac{z-i/2}{1+iz/2},\quad z\in U.
\]
Let
\[
W=\{(z^2,zb(z)):z\in U\}\subset \bC^2.
\]
The map $\Phi:U\to W$ defined as
\[
\Phi(z)=(z^2,zb(z)), \quad z\in U
\]
is easily verified to be a biholomorphism. Put $V=W\cap \bD^2$, so that $V=\Phi(\bD)$. It is straightforward to see that $\ol{V}$ intersects the boundary of $\bD^2$ only on $\bT^2$. Given $\lambda\in U$, if $z=\lambda^2$ and $w=\lambda b(\lambda)$, we find $|z|<4$ and
\[
4(w-z)^2-z(w+1)^2=0.
\]
Conversely, given $(z,w)\in \bC^2$ with $|z|<4$ such that $4(w-z)^2-z(w+1)^2=0$, let $\lambda=2\frac{w-z}{w+1}$. Then, $z=\lambda^2$ so $\lambda\in U$ and $w=\lambda b(\lambda)$.
We conclude that
\[
W=\left\{(z,w)\in \bC^2: 4(w-z)^2-z(w+1)^2=0,\quad |z|<4\right\}
\]
and $V$ is a distinguished variety in $\bD^2$. 
Finally, by Cartan's extension theorem (see for example \cite[Corollary 2.4.3]{Forstneric2011}), we may extend $\Phi^{-1}:W\to U$ to a holomorphic function $\Gamma:\{(z,w):|z|<4\}\to \bC$. \qed
\end{example}

\section{Agler--McCarthy varieties and the annihilator}\label{S:specsynth}
Assume that we are given a pair of commuting contractions that admits an Agler--McCarthy distinguished variety.  In this section, we aim to describe the variety using data determined by the pair. For this purpose, we start by recalling how Agler--McCarthy varieties arise in terms of dilation theory, following \cite{AM2005} and \cite{DS2017}.

\subsection{Hardy spaces preliminaries}
Let $H^2$ denote the Hardy space on the unit disc $\bD$. For $\lambda\in \bD$, we denote the associated reproducing kernel vector by $k_\lambda\in H^2$, where $k_\lambda(z)=(1-\ol{\lambda}z)^{-1}$ for each $z\in \bD$.
Recall that $H^\infty(\bD)$ and $H^\infty(\bD^2)$ denote the Banach algebras of bounded holomorphic functions on $\bD$ and $\bD^2$ (respectively), equipped with the supremum norm. It is known that both $H^\infty(\bD)$ and $H^\infty(\bD^2)$ are dual spaces, where weak-$*$ convergence of a sequence is characterized by bounded pointwise convergence.

We often will require vector-valued versions of these spaces. Let $E$ be a Hilbert space. Let $H^2(E)$ denote the Hilbert space of $E$-valued $H^2$ functions, which we identify with $H^2\otimes E$. The set $\{k_\lambda\otimes \xi:\lambda\in \bD, \xi\in E\}$ spans a dense subspace of $H^2\otimes E$. 

Given another Hilbert space $E'$,  we denote by $H^\infty(\bD; B(E',E))$ the Banach space of $B(E',E)$-valued bounded holomorphic functions on $\bD$.  For $$\Phi\in H^\infty(\bD;B(E',E)),$$ we let $M_\Phi:H^2\otimes E'\to H^2 \otimes E$ denote the corresponding multiplication operator acting as
\[
(M_\Phi f)(z)=\Phi(z)f(z), \quad z\in \bD, f\in H^2(E').
\] 
Then, \[\|\Phi\|_{H^\infty(\bD;B(E',E))}=\|M_\phi\|_{B(H^2\otimes E',H^2\otimes E)}.\] We say that $\Phi$ is \emph{inner} if $M_\Phi$ is an isometry. We have
\begin{equation}\label{Eq:reprod}
M_\Phi^* (k_\lambda\otimes \xi)=k_\lambda \otimes \Phi(\lambda)^*\xi
\end{equation}
for each $ \xi\in E$ and $\lambda\in \bD$. We need a well-known elementary fact about eigenvectors of multipliers.

\begin{lemma}\label{L:eigenvector}
Let $\lambda\in \bC$ and $f\in H^2(E)$. Then, the following statements are equivalent. 
\begin{enumerate}[{\rm (i)}]
\item $f$ is a non-zero element of $\ker ((M_z\otimes I)^*-\ol{\lambda}I)$.
\item We have $\lambda\in \bD$ and there is a non-zero $\xi\in E$ such that $f=k_\lambda\otimes \xi$.
\end{enumerate}
\end{lemma}
\begin{proof}
(ii)$\Rightarrow$(i): This follows directly from \eqref{Eq:reprod}. 

(i)$\Rightarrow$(ii): Assume that  $f$ is a norm $1$ element of $\ker ((M_z\otimes I)^*-\ol{\lambda}I)$. Write $f=\sum_{n=0}^\infty z^n \otimes \xi_n$ for some $\xi_n\in E$ satisfying $\sum_{n=0}^\infty \|\xi_n\|^2=1$. Using the equality $(M_z\otimes I)^* f=\ol{\lambda}f$, a straightforward comparison of Fourier coefficients reveals that $\xi_n=\ol{\lambda}\xi_{n-1}$ for each $n\geq 1$. Hence, $|\lambda|<1$ and
$$
f=\sum_{n=0}^\infty \ol{\lambda}^n z^n \otimes \xi_0=k_\lambda\otimes \xi_0.
$$
\end{proof}
Next, we record a crucial factorization property. Let $E,E'$ and $E''$ be Hilbert spaces. Let $\Phi\in H^\infty(\bD;B(E',E))$ and $\Psi\in H^\infty(\bD;B(E'',E))$. It follows from \cite[Theorem 8.57]{agler2002} that
\begin{equation}\label{Eq:Leech}
\ran M_\Phi\subset \ran M_\Psi \quad \text{if and only if} \quad M_\Phi=M_\Psi M_\Gamma
\end{equation}
for some $\Gamma\in H^\infty(\bD;B(E',E''))$. This will be used repeatedly below.

\subsection{An explicit isometric co-extension and functional calculus}\label{SS:coext}
Unless indicated otherwise, for the rest of this section, $T_1,T_2\in B(H)$ will be commuting contractions with finite defects, i.e. the ranges of $I-T_1T_1^*$ and $I-T_2T_2^*$ are both finite-dimensional. Let $d=\dim \ran( I-T_1T_1^*)$.  Assume in addition that both $T_1$ and $T_2$ are pure (or of class $C_{\cdot 0}$), namely
 \[
\lim_{n\to\infty} \|T_i^{*n}h\|=0,\quad h\in H, \quad i=1,2.
\]
By classical dilation theory \cite{nagy2010}, there is an isometry $J:H\to H^2\otimes \bC^d$ with the property that
\[
JT^*_1= (M^*_z\otimes 1) J.
\]
This is the usual minimal isometric co-extension of $T$. 

Next, it follows from \cite[Theorems 3.1 and 4.3]{DS2017} (see also \cite[Theorem 1.5]{DS2024} and \cite[Theorem 5.2]{sarkar2024}) that there is a rational matrix-valued pure inner function $\Psi\in H^\infty(\bD;B(\bC^d))$ such that 
\[
JT_2^*= M_\Psi^* J.
\]
Here, the fact that $\Psi$ is pure means that $\Psi$ has no unimodular eigenvalues in the interior of the disc. Thus
\begin{equation}\label{Eq:specPsi}
\sigma(\Psi(\lambda))\subset \bD, \quad \lambda\in \bD.
\end{equation}
In particular $\Psi(\lambda)$ is a pure contraction for each $\lambda\in \bD$.
If $f$ is a polynomial in two variables, then
\begin{equation}\label{Eq:functcalcpoly}
Jf(T_1,T_2)^*=f(M_z\otimes I,M_\Psi)^* J.
\end{equation}
We require an extension of this polynomial functional calculus. Let $f\in H^\infty(\bD^2)$ and let $\Delta_\Psi:\bD\to \bD\times B(\bC^d)$ be defined as
\[
\Delta_\Psi(z)=(z,\Psi(z)), \quad z\in \bD.
\]
 Define $f_z:\bD\to \bC$ as
\[
f_z(w)=f(z,w), \quad w\in \bD
\]
so that $f_z\in H^\infty(\bD)$. Thus, $f_z(\Psi(z))$ is well defined in view of \eqref{Eq:specPsi}. We obtain a function $f\circ \Delta_\Psi:\bD\to B(\bC^d)$ defined as
\[
(f\circ \Delta_\Psi)(z)=f_z(\Psi(z)),\quad z\in \bD.
\]
Since $\Psi(z)$ is a pure contraction for each $z\in \bD$, the usual $H^\infty(\bD)$-functional calculus implies that
$
\|(f\circ \Delta_\Psi)(z)\|\leq \|f_z\|
$
whence 
\begin{equation}\label{Eq:functcalcDelta}
\|f\circ \Delta_\Psi\|_{H^\infty(\bD;B(\bC^d))}\leq \|f\|_{H^\infty(\bD^2)}.
\end{equation}

\begin{lemma}\label{L:weak*}
The following statements hold.
\begin{enumerate}[{\rm (i)}]
\item If $f$ is a polynomial in two variables, then 
\[
f(M_z\otimes I,M_\Psi)=M_{f\circ \Delta_\Psi}.
\]
\item If $(f_n)$ is a sequence in $H^\infty(\bD^2)$ converging to $0$ in the weak-$*$ topology, then the sequence $(M_{f_n\circ \Delta_\Psi}^*)$ converges to $0$ in the strong operator topology of $B(H^2\otimes \bC^d)$.
\end{enumerate}
\end{lemma}
\begin{proof}
(i) Let $f$ be a polynomial in two variables. For $\lambda\in \bD$ and $\xi\in \bC^d$,  we have by virtue of \eqref{Eq:reprod} that
\begin{align*}
f(M_z\otimes I,M_\Psi)^*(k_\lambda\otimes \xi)&=k_\lambda \otimes f(\lambda,\Psi(\lambda))^*\xi=k_\lambda \otimes (f\circ \Delta_\Psi)(\lambda)^* \xi\\
&=M_{f\circ \Delta_\Psi}^*(k_\lambda\otimes \xi).
\end{align*}
By density, we conclude that $f(M_z\otimes I,M_\Psi)=M_{f\circ \Delta_\Psi}.$

(ii) 
The sequence $(f_n)$ is bounded in norm and thus so is $(M_{f_n\circ \Delta_\Psi})$ by \eqref{Eq:functcalcDelta}. Another density argument then reveals that it suffices to check that 
\[
\lim_{n\to\infty}\|M^*_{f_n\circ \Delta_\Psi} (k_\lambda\otimes \xi)\| =0
\]
for every $\lambda\in \bD$ and $\xi\in \bC^d$. By the calculation in the previous paragraph, this is equivalent to
\[
\lim_{n\to\infty}\|f_n(\lambda,\Psi(\lambda))^* \xi\| =0
\]
for each $\lambda\in \bD$. Fix $\lambda\in \bD$ and consider $g_n\in H^\infty(\bD)$ defined as
\[
g_n(z)=f_n(\lambda,z), \quad z\in \bD.
\]
By assumption on $(f_n)$, the sequence $(g_n)$ converges to $0$ in the weak-$*$ topology of $H^\infty(\bD)$. Furthermore, the spectrum of the matrix $\Psi(\lambda)$ is contained in $\bD$ by \eqref{Eq:specPsi}. Hence, the sequence of matrices $(g_n(\Psi(\lambda)))$ converges to $0$ in norm, and we find
\[
\lim_{n\to\infty}\|f_n(\lambda,\Psi(\lambda))^* \xi\| =\lim_{n\to\infty}\|g_n(\Psi(\lambda))^* \xi\|=0
\]
as required.
\end{proof}
In light of this lemma, we will use the notation $f(M_z\otimes I,M_\Psi)=M_{f\circ \Delta_\Psi}$ for $f\in H^\infty(\bD^2)$.
We may  define a unital, completely contractive, weak-$*$ continuous homomorphism $\rho:H^\infty(\bD^2)\to B(H)$ as
\begin{equation}\label{Eq:functcalcT}
\rho(f)=J^* f(M_z\otimes I,M_\Psi)J, \quad f\in H^\infty(\bD^2).
\end{equation}
We typically write $\rho(f)=f(T_1,T_2)$ for each $f\in H^\infty(\bD^2)$.  Note that this agrees with the usual definition when $f$ is a polynomial, by virtue of \eqref{Eq:functcalcpoly}. Approximating by polynomials, it is readily seen that
\begin{equation}\label{Eq:functcalcTcoext}
Jf(T_1,T_2)^*=f(M_z\otimes I,M_\Psi)^* J, \quad f\in H^\infty(\bD^2).
\end{equation}

\subsection{The distinguished variety}
We now arrive at a crucial step in our investigation. In the following statement, we retain the foregoing notation and assumptions.

\begin{theorem}\label{T:distvarspec}
Let
\[
V_\Psi=\{(z,w)\in \bD^2: \det(\Psi(z)-wI)=0\}.
\] 
Then, the following statements hold.
\begin{enumerate}[{\rm (i)}]
\item $V_\Psi$ is an Agler--McCarthy distinguished variety for the pairs $(M_z\otimes I,M_\Psi)$ and $(T_1,T_2)$.
\item Let $\delta\in H^\infty(\bD^2)$ be defined as
\[
\delta(z,w)=\det(\Psi(z)-w I), \quad (z,w)\in \bD^2.
\]
Then, $\delta(M_z\otimes I, M_\Psi)=0$ and $\delta(T_1,T_2)=0$.
\item There is a polynomial $p$ of two variables such that 
\[\ol{V_\Psi}=\{(z,w)\in \ol{\bD}^2: p(z,w)=0\},\]
$p(M_z\otimes I, M_\Psi)=0$ and $p(T_1,T_2)=0$.

\item  We have that $\ol{V_\Psi}=\sigma(M_z\otimes I,M_\Psi)$.
\end{enumerate}
\end{theorem}
\begin{proof}
(i) The fact that  $V_\Psi$ is an Agler--McCarthy distinguished variety for $(M_z\otimes I,M_\Psi)$ follows from standard estimates; see the proofs of \cite[Theorem 3.1]{AM2005} and \cite[Theorem 4.3]{DS2017}. The corresponding statement for $(T_1,T_2)$ is then a consequence \eqref{Eq:functcalcpoly}.

(ii) There are functions $\psi_0,\ldots,\psi_d\in H^\infty(\bD)$ such that
\[
\delta(z,w)=\sum_{n=0}^d \psi_n(z)w^n, \quad (z,w)\in \bD^2.
\]
For each fixed $z\in \bD$, the Cayley--Hamilton theorem implies that
\[
\sum_{n=0}^d \psi_n(z)\Psi(z)^n=0
\]
so that $\delta(M_z\otimes I, M_\Psi)=0$. By \eqref{Eq:functcalcT} we also obtain $\delta(T_1,T_2)=0$.

(iii) By Lemma \ref{L:polyconv}, there is a polynomial $p$ of two variables such that
\[
V_\Psi=\{(z,w)\in \bD^2:p(z,w)=0\} \qand \ol{V_\Psi}=\{(z,w)\in \ol{\bD}^2: p(z,w)=0\}.
\]
Combining (i) with \eqref{Eq:AM}, we see that 
\[
\|p(M_z\otimes I, M_\Psi)\|\leq \|p\|_{V_\Psi}=0
\]
so that $p(M_z\otimes I, M_\Psi)=0$. The fact that $p(T_1,T_2)=0$ then follows immediately from \eqref{Eq:functcalcT}.

(iv)  By (iii), we know that $p(M_z\otimes I,M_\Psi)=0$. It then follows from the spectral mapping theorem \cite[Theorem 7.7]{muller2007} that $\sigma(M_z\otimes I,M_\Psi)$ is contained in the zero set of $p$ inside of $\ol{\bD}^2$, that is $\ol{V_\Psi}$. Conversely, let $(\lambda,\mu)\in V_\Psi$. This means that there is a unit vector $\xi\in \bC^d$ such that $k_\lambda\otimes \xi$ is a joint eigenvector of $((M_z\otimes I)^*,M^*_\Psi)$ with joint eigenvalue $(\ol{\lambda},\ol{\mu})$. By  \cite[Corollary 9.14 and Remark 25.2]{muller2007}, we conclude that $(\lambda,\mu)\in \sigma(M_z\otimes I,M_\Psi)$. Hence, $V_\Psi\subset \sigma(M_z\otimes I,M_\Psi)$ so $\ol{V_\Psi}\subset \sigma(M_z\otimes I,M_\Psi)$ as well.
\end{proof}

Although we do not require it for our work below, we emphasize here that there is a somewhat explicit procedure for constructing $\Psi$ from $T_1$ and $T_2$ \cite{DS2017}. Hence, it is reasonable to ask if the variety $V_\Psi$, or perhaps some part thereof, can be identified as some canonical set associated to the pair $(T_1,T_2)$.

\subsection{Zero sets and annihilators}

The \emph{annihilator} of the pair $(T_1,T_2)$ is  defined as
 \[\Ann(T_1,T_2)=\{f\in H^\infty(\bD^2):f(T_1,T_2)=0\}.\]  
 This is easily seen to be a weak-$*$ closed ideal in $H^\infty(\bD^2)$. We seek to determine the precise relationship between $V_\Psi$ and $Z(\Ann(T_1,T_2))$, where for a subset $F\subset H^\infty(\bD^2)$, we write $Z(F)=\{(z,w)\in\bD^2:f(z,w)=0 \text{ for every } f\in F \}$. 

The starting point of our investigation is the following immediate consequence of Theorem  \ref{T:distvarspec}(iii):
\begin{equation}\label{Eq:ZannV}
Z(\Ann(T_1,T_2))\subset V_\Psi.
\end{equation}
%
Typically, $Z(\Ann(T_1,T_2))$ is much smaller than any spectral set, as can be seen upon considering the pair $\begin{bmatrix} 0 & 1 \\ 0 & 0 \end{bmatrix}$ and $\begin{bmatrix} 0 & 0 \\ 0 & 0 \end{bmatrix}$. We will show below that, under appropriate assumptions, $Z(\Ann(T_1,T_2))$ coincides with a certain natural subset of $V_\Psi$, defined by spectral conditions.

Retaining the notation used in Subsection \ref{SS:coext}, consider 
\[
K_\Psi=\bigcap_{f\in \Ann(T_1,T_2)} \ker f(M_z\otimes I,M_\Psi)^*.
\]
Note that for each $f\in \Ann(T_1,T_2)$, we have by virtue of \eqref{Eq:functcalcTcoext} that
\[
f(M_z\otimes I,M_\Psi)^*Jh=Jf(T_1,T_2)^*h=0, \quad h\in H
\]
which shows that 
\begin{equation}\label{Eq:JH}
JH\subset K_\Psi. 
\end{equation}
Moreover, it is clear by construction that $K_\Psi$ is co-invariant for $M_z\otimes I$ and $M_\Psi$. 

Define 
\begin{equation}\label{Eq:S}
S^\Psi_1=P_{K_\Psi} (M_z\otimes I)|_{K_\Psi} \qand S^\Psi_2 = P_{K_\Psi} M_\Psi|_{K_\Psi}
\end{equation}
so that
\begin{equation}\label{Eq:S*}
S^{\Psi*}_1=(M_z\otimes I)^*|_{K_\Psi} \qand S_2^{\Psi *}=M^*_\Psi|_{K_\Psi}.
\end{equation}
In particular, we conclude that $S_\Psi=(S^\Psi_1,S^\Psi_2)$ is a pair of commuting contractions. For ease of reference, we call  $(S_\Psi,K_\Psi)$ the \emph{constrained co-extension of $T$ corresponding to $\Psi$}.

Much like we did  in \eqref{Eq:functcalcTcoext}, we can define an $H^\infty(\bD^2)$-functional calculus for the pair $S_\Psi$, by setting
\begin{equation}\label{Eq:functcalcS}
f(S^\Psi_1,S^\Psi_2)^*=f(M_z\otimes I,M_\Psi)^*|_{K_\Psi}, \quad f\in H^\infty(\bD^2).
\end{equation}

We now want to establish some useful properties of $K_\Psi$, using the following general fact. 

\begin{proposition}\label{P:equaldefect}
Let $E\subset \bC^d$ be a subspace and let $\Theta\in H^\infty(\bD; B(E,\bC^d))$ be an inner function. Put $L=(\ran M_\Theta)^\perp$ and let $S=P_L (M_z\otimes I_{\bC^d})|_L$. Then, the following statements hold.
\begin{enumerate}[{\rm (i)}]
\item If $\Ann(S)$ is non-zero, then $E=\bC^d$.
\item If $E=\bC^d$, then the function $\alpha=\det \Theta$ is non-zero and satisfies
\[ 
\alpha H^\infty(\bD)\subset \Ann(S)\subset \{\phi\in  H^\infty(\bD):\phi^d\in \alpha H^\infty(\bD)\}.
\]
\end{enumerate}
\end{proposition}
\begin{proof}
(i) Let $\phi\in \Ann(S)$ be non-zero. Then
\[
0=\phi(S)^*=(M_\phi\otimes I)^*|_L
\]
so that $L\subset \ker (M_\phi\otimes I)^*$ or $\ran (M_\phi\otimes I)\subset \ran M_\Theta$ (note that $M_\Theta$ has closed range since it is an isometry). By \eqref{Eq:Leech}, there is $\Gamma\in H^\infty(\bD; B(\bC^d,E))$ such that $ \phi\otimes I=\Theta\Gamma$ on $\bD$. Since $\phi$ is non-zero, there is $z\in \bD$ such that $\phi(z)\neq 0$. The fact that $\phi(z)\otimes I=\Theta(z)\Gamma(z)$ implies in particular that $\Theta(z)$ is surjective, so that $\bC^d=E$ as desired.

(ii) By assumption, $\Theta(z)$ is a square matrix for each $z\in \bD$, so $\alpha$ is well defined. Since $\Theta$ is inner, there is $\zeta\in \bT$ such that $\Theta(\zeta)$ is a unitary matrix \cite[Proposition V.2.2]{nagy2010}, so that $\alpha$ is not the zero function. Next, let $\phi\in H^\infty(\bD)$ such that $\phi(S)=0$. As above, this is equivalent to the equality $\phi\otimes I=\Theta\Gamma$ for some $\Gamma\in H^\infty(\bD; B(\bC^d))$. Taking determinants we find $\phi^d\in \alpha H^\infty(\bD)$.  Thus,
\[
\Ann(S)\subset \{\phi\in  H^\infty(\bD):\phi^d\in \alpha H^\infty(\bD)\}.
\]
Next, let $\Gamma:\bD\to B(\bC^d)$ be such that $\Gamma(z)$ is the algebraic adjoint of $\Theta(z)$ for each $z\in \bD$. Then, $\Gamma\in H^\infty(\bD;B(\bC^d))$ and $\alpha \otimes I=\Theta\Gamma$. This shows that $$(M_\alpha \otimes I)\subset \ran M_\Theta=L^\perp,$$ so that $\alpha(S)=0$. We conclude that
\[ 
\alpha H^\infty(\bD)\subset \Ann(S).
\]
\end{proof}

We can now record the required properties of $K_\Psi$, by specializing the previous result.

\begin{corollary}\label{C:detK}
There is a subspace $E\subset \bC^d$ and an inner function $$\Theta\in H^\infty(\bD;B(E,\bC^d))$$ with the following properties.
\begin{enumerate}[{\rm (i)}]
\item $K_\Psi=(\ran M_\Theta)^\perp$.
\item Given $\xi\in \bC^d$ and $\lambda\in \bD$, the vector $k_\lambda\otimes \xi$ belongs to $K_\Psi$ if and only if $\Theta(\lambda)^*\xi=0$.
\item If $\Ann(S^\Psi_1)$ is non-zero, then $\bC^d=E$ and
\[ 
\alpha H^\infty(\bD)\subset \Ann(S_1^\Psi)\subset \{\phi\in  H^\infty(\bD):\phi^d\in \alpha H^\infty(\bD)\}
\]
where $\alpha=\det \Theta$.
\end{enumerate}
\end{corollary}
\begin{proof}
Since $K_\Psi$ is co-invariant for $M_z\otimes I$, by the Beurling--Lax--Halmos theorem there is a subspace $E\subset \bC^d$ and an inner function $\Theta\in H^\infty(\bD;B(E,\bC^d))$ such that $K_\Psi=\ker M_\Theta^*=(\ran M_\Theta)^\perp$, which gives (i). Property (ii) then follows immediately from \eqref{Eq:reprod}, while (iii) follows from Proposition \ref{P:equaldefect}.
\end{proof}

Another key property of $(S^\Psi_1,S^\Psi_2)$ is the following.

\begin{lemma}\label{L:AnnS}
We have $\Ann(S^\Psi_1,S^\Psi_2)=\Ann(T_1,T_2)$ and $\Ann(S^\Psi_1)=\Ann(T_1)$.
\end{lemma}
\begin{proof}
For $f\in H^\infty(\bD^2)$, by \eqref{Eq:functcalcTcoext} and \eqref{Eq:functcalcS} we have
\[Jf(T_1,T_2)^*=f(M_z\otimes I,M_\Psi)^* J=f(S^\Psi_1,S^\Psi_2)^* J
\]
where we used that $JH\subset K_\Psi$ by \eqref{Eq:JH}. In particular, this shows that $f(T_1,T_2)=0$ if $f(S^\Psi_1,S^\Psi_2)=0$, so that $\Ann(S^\Psi_1,S^\Psi_2)\subset \Ann(T_1,T_2)$. Conversely, let $f\in \Ann(T_1,T_2)$. Then, $K_\Psi\subset \ker f(M_z\otimes I,M_\Psi)^*$ by construction, so that
\[
f(S^\Psi_1,S^\Psi_2)^*=f(M_z\otimes I,M_\Psi)^*|_{K_\Psi}=0
\]
by \eqref{Eq:functcalcS}. We conclude that $f\in\Ann(S^\Psi_1,S^\Psi_2)$. We have shown that $\Ann(S^\Psi_1,S^\Psi_2)=\Ann(T_1,T_2)$.

Finally, let $g\in H^\infty(\bD)$. Define $f\in H^\infty(\bD^2)$ as
\[
f(z,w)=g(z),\quad (z,w)\in \bD^2.
\]
Then $g\in \Ann(T_1)$ is equivalent to $f\in \Ann(T_1,T_2)$, and likewise $g\in \Ann(S^\Psi_1)$ is equivalent to $f\in \Ann(S^\Psi_1,S^\Psi_2)$. By the first part of the proof, we conclude that $\Ann(S^\Psi_1)=\Ann(T_1)$.
\end{proof}

Next, we define $\Omega_\Psi\subset \bC^2$ to be the set of points $(\lambda,\mu)\in \bC^2$ such that $(\ol{\lambda},\ol{\mu})$ is a joint eigenvalue of $(S_1^{\Psi*},S_2^{\Psi*})$. This set has the following remarkable property, which is one of the main results of this section.

\begin{theorem}\label{T:ZAnn}
Let $(T_1,T_2)$ be a pair of pure commuting contractions on some Hilbert space, both with finite defect.  Let $(S_\Psi,K_\Psi)$ be the constrained  co-extension of $(T_1,T_2)$ corresponding to some rational matrix-valued pure inner function $\Psi\in H^\infty(\bD;B(\bC^d))$. Then, we have $Z(\Ann(T_1,T_2))=\Omega_\Psi$.
\end{theorem}
\begin{proof}
Let $(\lambda,\mu)\in \Omega_\Psi$. By definition, this means that there is a unit vector $\xi\in K_\Psi$ such that
\[
S_1^{\Psi*}\xi=(M_z\otimes I)^* \xi=\ol{\lambda}\xi \qand S_2^{\Psi*}\xi=M_\Psi^*\xi=\ol{\mu}\xi.
\]
Let $f\in \Ann(T_1,T_2)$. Then, $f\in \Ann(S^\Psi_1,S^\Psi_2)$ by Lemma \ref{L:AnnS} so
\[
0=f(S^\Psi_1,S^\Psi_2)^*\xi=\ol{f(\lambda,\mu)}\xi
\]
which implies that $f(\lambda,\mu)=0$. We conclude that $\Omega_\Psi\subset Z(\Ann(T_1,T_2))$.

Conversely, let $(\lambda,\mu)\in Z(\Ann(T_1,T_2))$. Define $\delta\in H^\infty(\bD^2)$ as
\[
\delta(z,w)=\det(\Psi(z)-w I), \quad (z,w)\in \bD^2.
\]
Then, $\delta\in \Ann(T_1,T_2)$ by Theorem \ref{T:distvarspec}. We conclude that $\delta(\lambda,\mu)=0$, so $\ol{\mu}$ is an eigenvalue of $\Psi(\lambda)^*$. Hence, there is a unit vector $\eta\in \bC^d$ such that $\Psi(\lambda)^*\eta=\ol{\mu}\eta$. In turn, this implies that $k_\lambda\otimes \eta$ is a joint eigenvector of $M_z^*\otimes I$ and $M_\Psi^*$, with corresponding joint eigenvalue $(\ol{\lambda},\ol{\mu})$. To finish the proof, it thus suffices to verify that $k_\lambda\otimes \eta\in K_\Psi$. 
To see this, let $f\in \Ann(T_1,T_2)$. Then, 
\[
f(M_z\otimes I,M_\Psi)^*(k_\lambda\otimes \eta)=\ol{f(\lambda,\mu)} (k_\lambda \otimes \eta)=0
\]
since $(\lambda,\mu)\in Z(\Ann(T_1,T_2))$. Hence, $k_\lambda\otimes \eta$ belongs to
\[ \bigcap_{f\in \Ann(T_1,T_2)}\ker f(M_z\otimes I,M_\Psi)^*=K_\Psi.\qedhere\]
\end{proof}

For future technical use, we now describe the projection of $\Omega_\Psi$ onto the first component of $\bC^2$.

\begin{theorem}\label{T:ZAnn1}
Let $T=(T_1,T_2)$ be a pair of pure commuting contractions on some Hilbert space, both with finite defect. Let $(S_\Psi,K_\Psi)$ be the constrained co-extension of $(T_1,T_2)$ corresponding to some rational matrix-valued pure inner function $\Psi\in H^\infty(\bD;B(\bC^d))$.  Let $\Omega_1\subset \bC$ denote the projection of $\Omega_\Psi$ onto the first component of $\bC^2$. If $\Ann(T_1)$ is non-zero,  then $\Omega_1=Z(\Ann(T_1))$.
\end{theorem}
\begin{proof}
Let $f\in \Ann(T_1)$ and $(\lambda,\mu)\in \Omega_\Psi$. Let $F(z,w)=f(z)$ for $(z,w)\in \bD^2$. Then $F\in \Ann(T_1,T_2)$ so that $f(\lambda)=F(\lambda,\mu)=0$ by Theorem \ref{T:ZAnn}.  This shows that $\Omega_1\subset Z(\Ann(T_1))$.

Conversely, let $\lambda\in Z(\Ann(T_1))$.  By Lemma \ref{L:AnnS}, we see that $\Ann(S^\Psi_1)=\Ann(T_1)$ is non-zero, so Corollary \ref{C:detK} implies that there is an inner function $\Theta\in H^\infty(\bD;B(\bC^d))$ such that $K_\Psi=\ker M_\Theta^*$ and $\det\Theta\in \Ann(T_1)$. Hence $\det\Theta(\lambda)=0$ and there is a non-zero vector $\xi\in \bC^d$ such that $\Theta(\lambda)^*\xi=0$. In particular, $k_\lambda\otimes \xi\in K_\Psi$. 

Next, let $G\subset \Ann(T_1,T_2)$ be a finite subset. By definition, we then have \[K_\Psi\subset \bigcap_{\omega\in G} \ker \omega(M_z\otimes I, M_\Psi)^*\] so
\[
0=\omega(M_z\otimes I, M_\Psi)^* (k_\lambda \otimes \xi)=k_\lambda\otimes \omega(\lambda, \Psi(\lambda))^*\xi
\]
for each $\omega\in G$.
Therefore   $0$ is a joint eigenvalue of the matrices  $$\{\omega(\lambda,\Psi(\lambda))^*:\omega\in G\}.$$
For each $\omega\in G$, define $\tilde\omega\in H^\infty(\bD^2)$ as
\[
\tilde\omega(z,w)=\ol{\omega(\ol{z},\ol{w})}, \quad (z,w)\in \bD^2
\]
Then, it is easily verified that 
\[
\omega(\lambda,\Psi(\lambda))^*=\tilde \omega(\ol{\lambda},\Psi(\lambda)^*), \quad \omega\in G.
\]
By \eqref{Eq:specPsi}, the spectrum of $\Psi(\lambda)^*$ is contained in $\bD$, so each function $\tilde\omega(\ol{\lambda},\cdot)$ is holomorphic in a neighbourhood of the spectrum. Thus we may apply \cite[Theorem 7.7]{muller2007} to  find $\mu_G\in \bD$ such that $\ol{\mu_G}$ is an eigenvalue of $\Psi(\lambda)^*$ and \[
\ol{\omega(\lambda,\mu_G)}=\tilde \omega(\ol{\lambda},\ol{\mu_G})=0
\]
 for each $\omega\in G$.  

Observe finally that the net $(\mu_G)$ lies in the spectrum of $\Psi(\lambda)$, which is compact. Hence, upon passing to a cofinal subnet if necessary, we may assume that $(\mu_G)$ converges to some $\mu\in \bD$ in the spectrum of $\Psi(\lambda)$. By construction, we find that $\omega(\lambda,\mu)=0$ for each $\omega\in \Ann(T_1,T_2)$, or  $(\lambda,\mu)\in Z(\Ann(T_1,T_2))$, so we conclude that $\lambda\in \Omega_1$ by Theorem \ref{T:ZAnn}, and the proof is complete.
\end{proof}

\subsection{Boundary behaviour and the support}

Note that by virtue of \cite[Corollary 9.14 and Remark 25.2]{muller2007}, the set $\Omega_\Psi$ is contained in the Taylor spectrum of $(S^\Psi_1,S^\Psi_2)$. In light of Theorem \ref{T:ZAnn}, one may then wonder if there is a relationship between the full Taylor spectrum and $\Ann(T_1,T_2)$. 

To address this issue, we require the following notion, inspired by \cite{CT2023}. Let $F\subset H^\infty(\bD^2)$ be an ideal. The  {\em support} of $F$, denoted by $\supp(F)$, is the subset of $\mathbb{C}^2$ that consists of all points $(\lambda,\mu)$ such that 
\[
1\notin F+(z-\lambda)H^\infty(\bD^2)+(w-\mu)H^\infty(\bD^2).
\]
We require a few basic properties, adapting arguments from  \cite{CT2023}.
	 
	 \begin{lemma}\label{L:supp}
	 	Let $F\subset H^{\infty}(\mathbb{D}^2)$ be an ideal. Then, $\operatorname{supp}(F) \subset \overline{\mathbb{D}^2}$ and 
	 	$\operatorname{supp}(F) \cap \mathbb{D}^2= Z(F)$.
	 \end{lemma}
	 \begin{proof}
Let $(\lambda,\mu)\in \mathbb{C}^2$ lie outside $\overline{\mathbb{D}^2}$. Without loss of generality we can assume that $|\lambda|>1$, as the case where $|\mu|>1$ is symmetric. Consider the function $g\in H^\infty(\bD^2)$ defined as
\[
g(z,w)=1/(z-\lambda), \quad (z,w)\in \bD^2.
\]
Then, $1=(z-\lambda)g$ so that $(\lambda,\mu)\notin \supp(F)$ by definition. This shows that $\operatorname{supp}(F) \subset \overline{\mathbb{D}^2}$.

Next, let $(\lambda,\mu)\in \mathbb{D}^2$ lie outside of $Z(F)$. Then, there is $f\in F$ such that $f(\lambda,\mu)\neq 0$. Define $f_1,f_2:\bD^2\to \bC$ as
\[
f_1(z,w)=\frac{f(z,\mu)-f(\lambda,\mu)}{z-\lambda} \qand f_2(z,w)=\frac{f(z,w)-f(z,\mu)}{w-\mu}
\]
for each $z,w\in \bD$. It is straightforward to verify that $f_1,f_2\in H^\infty(\bD^2)$ and that
	 	$$	f(\lambda,\mu)= f-(z-\lambda)f_1-(w-\mu)f_2.$$
	 	Since $f(\lambda,\mu)$ is non-zero, it follows that $(\lambda,\mu)\notin\supp(F)$. Conversely, assume that $(\lambda,\mu)\in \bD^2$  lies outside of  $\supp(F)$. By the definition,  there are two functions $f_1, f_2$ in $H^{\infty}(\mathbb{D}^2)$ such that 
	 	$$	1+ (z-\lambda )f_1+ (w-\mu) f_2\in F	.$$
	 	In particular, this shows that $(\lambda,\mu)\notin Z(F)$.
	 \end{proof}
	 
\begin{lemma}\label{L:suppTaylor}
Let $(T_1,T_2)$ be a pair of commuting contractions on some Hilbert space admitting an $H^\infty(\bD^2)$-functional calculus. Then, 
\[\sigma(T_1,T_2)\subset \supp(\Ann(T_1,T_2)).\]
\end{lemma}
\begin{proof}
Let $(\lambda,\mu)\in \bC^2$ be a point outside $\supp(\Ann(T_1,T_2))$. Then, there are $\phi\in \Ann(T_1,T_2)$ and $f_1,f_2\in H^\infty(\bD^2)$ such that 
\[
1=\phi+(z_1-\lambda)f_1+(w-\mu)f_2.
\]
Applying the functional calculus we find
\[
I=(T_1-\lambda I)f_1(T_1,T_2)+(T_2-\mu I)f_2(T_1,T_2)
\]
so that $(\lambda,\mu)\notin \sigma(T_1,T_2)$ by \cite[Proposition 25.3]{muller2007}.
\end{proof}
	 
	 We can now give a refinement of Theorem \ref{T:ZAnn} which  contains further information about the boundary behaviour of the Taylor spectrum of the constrained co-extension.

	 \begin{theorem}\label{T:supp}
Let $(T_1,T_2)$ be a pair of pure commuting contractions on some Hilbert space, both with finite defect.  Let $(S_\Psi,K_\Psi)$ be the constrained  co-extension of $(T_1,T_2)$ corresponding to some rational matrix-valued pure inner function $\Psi\in H^\infty(\bD;B(\bC^d))$. Then, we have 
\[
\sigma(S^\Psi_1,S^\Psi_2)\subset \supp(\Ann(T_1,T_2))\subset \ol{V_\Psi}=\supp(\Ann(M_z\otimes I,M_\Psi))
\]
and
\[
\sigma(S^\Psi_1,S^\Psi_2)\cap \bD^2=\supp(\Ann(T_1,T_2))\cap \bD^2=Z(\Ann(T_1,T_2)).
\]
	 
	 \end{theorem}
	 
	 \begin{proof}
	 By Theorem \ref{T:distvarspec}, there is a polynomial $p\in \Ann(T_1,T_2)\cap \Ann(M_z\otimes I,M_\Psi)$ such that
	 \[
	 \ol{V_\Psi}=\{(z,w)\in \ol{\bD}^2:p(z,w)=0\}.
	 \]
	 Let $(\lambda,\mu)\in \ol{\bD}^2\setminus \ol{V_\Psi}$. Then $p(\lambda,\mu)\neq 0$. Observe that there are polynomials $q_1,q_2$ such that
	 \[
	 p-p(\lambda,\mu)=(z-\lambda)q_1+(w-\mu)q_2
	 \]
so that $(\lambda,\mu)$ lies outside $\supp(\Ann(T_1,T_2))$ and $\supp(\Ann(M_z\otimes I,M_\Psi))$. Since we know that the support always lies in $\ol{\bD}^2$  by Lemma \ref{L:supp}, we conclude that $\supp(\Ann(T_1,T_2))\subset \ol{V_\Psi}$ and $\supp(\Ann(M_z\otimes I,M_\Psi))\subset \ol{V_\Psi}$.

Next, we apply Lemma \ref{L:suppTaylor} to see that
\[
\sigma(S^\Psi_1,S^\Psi_2)\subset \supp(\Ann(S^\Psi_1,S^\Psi_2))\]
and
\[\sigma(M_z\otimes I, M_\Psi)\subset \supp(\Ann(M_z\otimes I,M_\Psi)).
\]
Applying Theorem \ref{T:distvarspec} and Lemma \ref{L:AnnS}, these become
\[
\sigma(S^\Psi_1,S^\Psi_2)\subset \supp(\Ann(T_1,T_2))\]
and
\[\ol{V_\Psi}\subset \supp(\Ann(M_z\otimes I,M_\Psi)).
\]
Hence, we have shown
\[
\sigma(S^\Psi_1,S^\Psi_2)\subset \supp(\Ann(T_1,T_2))\subset \ol{V_\Psi}=\supp(\Ann(M_z\otimes I,M_\Psi))
\]
Finally, it follows from Lemma \ref{L:supp} that $$\supp(\Ann(T_1,T_2))\cap\bD^2=Z(\Ann(T_1,T_2)).$$ In turn, Theorem \ref{T:ZAnn} implies that $$\supp(\Ann(T_1,T_2))\cap\bD^2=\Omega_\Psi$$ which is contained in $\sigma( S^\Psi_1,S^\Psi_2)\cap \bD^2$ as explained at the beginning of this subsection. This completes the proof.
\end{proof}
	
Notably, when combined with Theorem \ref{T:distvarspec}, the previous result implies that for the pair $(M_z\otimes I,M_\Psi)$, the Taylor spectrum coincides with the support of the annihilating ideal.

\subsection{The annihilator as a vanishing ideal}

Motivated by Theorem \ref{T:ZAnn}, another natural question is whether $\Ann(T_1,T_2)$ can be described in terms of functions vanishing on $\Omega_\Psi$. Given a subset $X\subset \bD^2$, we let $I(X)\subset H^\infty(\bD^2)$ consist of those functions  vanishing on $X$.

\begin{theorem}\label{T:specsynth}
Let $(T_1,T_2)$ be a pair of pure commuting contractions on some Hilbert space, both with finite defect. Let $(S_\Psi,K_\Psi)$ be the constrained co-extension of $(T_1,T_2)$ corresponding to some rational matrix-valued pure inner function $\Psi\in H^\infty(\bD;B(\bC^d))$. Assume that $\Ann(T_1)$ is non-zero. Consider the following statements.
\begin{enumerate}[{\rm (i)}]
\item If $\Theta\in H^\infty(\bD;B(\bC^d))$ is an inner function such that $K_\Psi=(\ran M_\Theta)^\perp$, then $\det\Theta$ is a Blaschke product with simple roots.
\item The joint eigenvectors of $S_1^{\Psi *}$ and $S_2^{\Psi *}$ span a dense subspace in $K_\Psi$.
\item $\Ann(T_1,T_2)=I(Z(\Ann(T_1,T_2)))=I(\Omega_\Psi)$.
\item $\Ann(T_1)$ is a radical ideal in $H^\infty(\bD)$.
\item $\Ann(T_1)$ is generated by a Blaschke product with simple roots.
\end{enumerate}
Then, we have
\[
{\rm (i)}\Rightarrow {\rm (ii)}\Rightarrow{\rm (iii)}\Rightarrow{\rm (iv)}\Leftrightarrow{\rm (v)}.
\]
\end{theorem}
\begin{proof}
(ii)$\Rightarrow$ (iii): By Theorem \ref{T:ZAnn}, we have $Z(\Ann(T_1,T_2))=\Omega_\Psi$ so certainly \[I(Z(\Ann(T_1,T_2)))=I(\Omega_\Psi).\]
Plainly, it always holds that \[\Ann(T_1,T_2)\subset I(Z(\Ann(T_1,T_2))).
\]
Let $f\in I(\Omega_\Psi)$. We claim that $f\in \Ann(T_1,T_2)$. By Lemma \ref{L:AnnS} and our assumption, to establish this claim it suffices to fix a joint eigenvector  $h\in K_\Psi$ for $S_1^{\Psi *}$ and $S_2^{\Psi *}$, and to show that $$f(S^\Psi_1,S^\Psi_2)^* h=0.$$ Now, $h$ must be an eigenvector for $(M_z\otimes I)^*$, so Lemma \ref{L:eigenvector} implies there are $\lambda\in \bD$ and a non-zero $\xi\in \bC^d$ such that $h=k_\lambda\otimes \xi$. Since $h$ is also an eigenvector for $M_\Psi^*$, there is $\mu$ in the spectrum of $\Psi(\lambda)^*$ such that $\Psi(\lambda)^* \xi=\ol{\mu}\xi$.  We see that $\mu\in \bD$ by \eqref{Eq:specPsi}.

Next, the fact that $h\in K_\Psi$ means, by definition of $K_\Psi$, that for $g\in \Ann(T_1,T_2)$ we have
\[
0=g(M_z\otimes I, M_\Psi)^*(k_\lambda\otimes \xi)=\ol{g(\lambda,\mu)}(k_\lambda\otimes \xi)
\]
so that $(\lambda,\mu)\in Z(\Ann(T_1,T_2))$. By Theorem \ref{T:ZAnn}, we infer $(\lambda,\mu)\in \Omega_\Psi$ so $f(\lambda,\mu)=0$ by choice of $f$. This means that
\[
f(S^\Psi_1,S^\Psi_2)^* h=f(M_z\otimes I,M_\Psi)^* h=\ol{f(\lambda,\mu)}h=0
\]
and the claim follows.

(iii)$\Rightarrow$ (iv): This is elementary. Let $f\in H^\infty(\bD)$ with the property that $f^n\in \Ann(T_1)$ for some $n\geq 1$. Define $F(z,w)=f(z)$ for each $(z,w)\in \bD^2$. Then, $$F^n\in \Ann(T_1,T_2)=I(\Omega_\Psi).$$ This shows that $F^n$ vanishes on $\Omega_\Psi$, and thus so does $F$. Hence, $F\in I(\Omega_\Psi)=\Ann(T_1,T_2)$, so indeed $f\in \Ann(T_1)$. 

(iv)$\Leftrightarrow$ (v): The ideal $\Ann(T_1)$ is closed in the weak-$*$ topology of $H^\infty(\bD)$, so there is an inner function $\omega\in H^\infty(\bD)$ with $\Ann(T_1)=\omega H^\infty(\bD)$ \cite{bercovici1988}. If $\omega$ is not a Blaschke product with simple roots, it is easy to construct $\alpha\notin \Ann(T_1)$ such that $\alpha^2\in \Ann(T_1)$. Conversely, the fact that $\Ann(T_1)$ is radical when $\omega$ is a Blaschke product with simple roots follows from \cite[Theorem 4.2]{clouatre2015CB}.

(i)$\Rightarrow$(ii): Let $L\subset K_\Psi$ denote the closed subspace spanned by the joint eigenvectors of $S_1^{\Psi *}$ and $S_2^{\Psi *}$. Then, $L$ is co-invariant for $M_z\otimes I$, so by the Beurling--Lax--Halmos theorem there is a subspace $E\subset \bC^d$ along with an inner function $\Gamma\in H^\infty(\bD;B(E,\bC^d))$ such that $L=(\ran M_\Gamma)^\perp$. By assumption, there is a non-zero $f\in \Ann(T_1)$. Then $f\in \Ann(S^\Psi_1)$ by Lemma \ref{L:AnnS} so in particular $f$ annihilates 
$$P_L S_1^\Psi|_L=P_L (M_z\otimes I)|_L.$$ Applying Proposition \ref{P:equaldefect}, we see that $E=\bC^d$.

Next, we see from Corollary \ref{C:detK} that there is an inner function $\Theta\in H^\infty(\bD;B(\bC^d))$ such that $K_\Psi=(\ran M_\Theta)^\perp$ and 
\[ 
\alpha H^\infty(\bD)\subset \Ann(S^\Psi_1)\subset \{\phi\in  H^\infty(\bD):\phi^d\in \alpha H^\infty(\bD)\}.
\]
where $\alpha=\det\Theta$. Our assumption then says that $\alpha$ must be a Blascke product with simple roots. By Lemma \ref{L:AnnS}, it follows that $\alpha\in \Ann(T_1)$.  Let $\lambda$ be a root of $\alpha $. Then, $\lambda\in Z(\Ann (T_1))$, so we may apply Theorem \ref{T:ZAnn1}  to infer that there is $\mu\in \bD$ with the property that $(\lambda,\mu)\in \Omega_\Psi$. In particular, there is $h\in L=\ker M_\Gamma^*$ such that $(M_z\otimes I)^* h=\ol{\lambda} h.$ We conclude from Lemma \ref{L:eigenvector} that there is a non-zero $\xi\in \bC^d$ such that $h=k_\lambda\otimes \xi$. Then,
\[
0=M_\Gamma^* h=k_\lambda\otimes \Gamma(\lambda)^*\xi
\]
so that $\Gamma(\lambda)$ is not invertible and $\det\Gamma(\lambda)=0$. We conclude that $\det \Gamma$ vanishes at all the roots of the simple Blaschke product $\alpha$, so we must have $\det\Gamma=\alpha\phi$ for some $\phi\in H^\infty(\bD)$. 

Conversely, the inclusion $L\subset K_\Psi$ implies $\ran M_\Theta\subset \ran M_\Gamma$, so by \eqref{Eq:Leech} there is $\Delta\in H^\infty(\bD;B(\bC^d))$ such that $\Theta=\Gamma \Delta$. In particular, $\alpha=\det\Theta =(\det \Gamma)(\det \Delta)$. Combined with the conclusion of the previous paragraph, we infer that $\det\Delta$ is invertible in $H^\infty(\bD)$, so that $ M_\Delta$ is invertible by the adjoint formula for the inverse of a matrix. Thus, $$L^\perp=\ran M_\Gamma=\ran M_\Theta=K_\Psi^\perp$$ and therefore $L=K_\Psi$.
\end{proof}

Condition (v) above combined with Corollary \ref{C:detK} gives 
\[ 
\alpha H^\infty(\bD)\subset \Ann(S^\Psi_1)\subset \{\phi\in  H^\infty(\bD):\phi^d\in \alpha H^\infty(\bD)\}.
\]
where $\alpha=\det\Theta$. Although this is somewhat reminiscent of condition (i), it falls short. In fact, even the implication (v)$\Rightarrow$(iii) can fail dramatically, as we illustrate below.
\begin{example}\label{E:counter}
On the Hilbert space $H=\bC^2$, consider the commuting pair of pure contractions
\[
T_1=0 \qand T_2=\begin{bmatrix} 0 & 1 \\ 0 & 0 \end{bmatrix}.
\]
In this case, $d=2$. 
We find that $\Ann(T_1)$ is generated by the simple Blaschke product $z$. On the other hand, using that $z,w^2\in \Ann(T_1,T_2)$ we see that 
\[
Z(\Ann(T_1,T_2))=\{(0,0)\}
\]
so that the function $w$ vanishes on $Z(\Ann(T_1,T_2))$ but lies outside of $\Ann(T_1,T_2)$.
\qed 
\end{example}

\textbf{Acknowledgments:} The authors are grateful to an anonymous referee for numerous insightful comments and observations that improved the paper. The second author would like to thank Bata Krishna Das and Haripada Sau for helpful discussions.

\bibliography{AMspecvar}  
	
\end{document}